\documentclass[12pt]{amsart}
\usepackage{amssymb}
\usepackage{mathtools}
\usepackage[english]{babel}
\usepackage[utf8]{inputenc}
\usepackage[T1]{fontenc}
\usepackage{mathptmx}

\usepackage{epsfig}
\usepackage{tikz}
\usepackage{amsmath,amsfonts,amssymb,amsthm}
\usepackage{mathrsfs}
\usepackage{enumitem}
\usepackage[all]{xy}
\usepackage{graphicx}
\usepackage{latexsym}
\usepackage{verbatim}
\usepackage{hyperref}

\numberwithin{equation}{section}
\newtheorem{theorem}{Theorem}[section]

\newtheorem{proposition}[theorem]{Proposition}
\newtheorem{corollary}[theorem]{Corollary}
\newtheorem{question}[theorem]{Question}

\theoremstyle{definition}
\newtheorem{definition}[theorem]{Definition}
\newtheorem{example}[theorem]{Example}
\newtheorem{property}[theorem]{Property}

\theoremstyle{remark}
\newtheorem{remark}[theorem]{Remark}

\numberwithin{equation}{section}

\title{Holomorphic semigroups of finite shift in the unit disc}
\author[D. Cordella]{Davide Cordella$^\dag$}

\address{D. Cordella: Dipartimento di Matematica, Universit\`a di Roma `Tor Vergata', Via della Ricerca Scientifica 1, 00133, Roma, Italia.} 
\email{cordella@mat.uniroma2.it}


\thanks{$^\dag$Partially supported by PRIN {\sl Real and Complex
		Manifolds: Topology, Geometry and holomorphic dynamics} n.2017JZ2SW5 and  by the MIUR Excellence Department Project awarded to the
	Department of Mathematics, University of Rome Tor Vergata, CUP E83C18000100006}

\begin{document}
	
\begin{abstract}
	We give three necessary and sufficient conditions so that a parabolic holomorphic semigroup $(\phi_t)$ in the unit disc is of finite shift. One is in terms of the asymptotic behavior of speeds of convergence, the second one is related to the hyperbolic metric of its Koenigs domain $\Omega$ and the latter one deals with Euclidean properties of $\Omega$.
\end{abstract}

\maketitle

\tableofcontents

\section{Introduction}
The study of continuous semigroups of holomorphic self-maps in the unit disc $\mathbb{D}:=\{z\in\mathbb{C}:|z|<1\}$ (or simply semigroups in $\mathbb D$ for short)
\[
[0,+\infty)\ni t\longmapsto \phi_t\in\mathsf{Hol}(\mathbb{D},\mathbb{D})
\]
began in the nineteenth century and nowadays provides an impressive theory which is of great interest for its several applications and also from a theoretical point of view, as it combines several aspects from analysis, geometry and dynamics. For more details on this subject, see for instance the recent monograph \cite{BCDbook}, the first chapter in \cite{Ababook89} and also \cite{EliShobook10,Shobook01}.

It is well known that, if for some $t_0>0$ the map  $\phi_{t_0}$ has a fixed point in $\mathbb D$, then the $\phi_t$'s have a common fixed point in $\mathbb D$. In this case, the semigroup is called an \emph{elliptic} semigroup.

In this paper the focus is on \emph{non-elliptic} semigroups, i.e. the case in which the maps $\phi_t$ have no fixed points in $\mathbb D$. Then there exist a unique point $\tau\in \partial\mathbb D$, called the \emph{Denjoy-Wolff point}, such that all the orbits $t\mapsto \phi_t(z)$ where $z\in\mathbb{D}$ converge to $\tau$. 

Every non-elliptic semigroup $(\phi_t)$ admits an holomorphic model -- unique up to translations -- given by a univalent function $h:\mathbb D \to \mathbb C$ (the so-called Koenigs function) such that  $h(\mathbb D)+it\subset h(\mathbb D)$  and $h\circ \phi_t(z)=h(z)+it$ for all $t\geq 0$ and $\bigcup_{t\geq 0}(h(\mathbb D)-it)$ is either a strip, a right (or left) vertical half-plane or the whole complex plane $\mathbb C$. In the first case the semigroup is called \emph{hyperbolic}, in the second \emph{parabolic of positive hyperbolic step} and in the third case \emph{parabolic of zero hyperbolic step} (see, e.g., \cite{BCDbook} for details, in particular Theorem 9.3.5 and Proposition 9.3.10 for uniqueness).

Let $(\phi_t)$ be a non-elliptic semigroup in $\mathbb D$. The Denjoy-Wolff Theorem (\cite[Theorem 1.8.4]{BCDbook}) states that the horocycles at $\tau$, which are discs contained in $\mathbb{D}$  tangent to $\partial \mathbb D$ at $\tau$, are invariant  for $(\phi_t)$. More precisely, given a horocycle $\mathcal{E}$ at $\tau$, for every $z\in \mathcal{E}$, the orbit $[0,+\infty)\ni t\mapsto \phi_t(z)$ belongs to $\mathcal{E}$. However, if $\mathcal{E}'\subsetneq \mathcal{E}$ is another horocycle and $z\in \mathcal{E}\setminus \mathcal{E}'$, it might happen that $\phi_{t_0}(z)\in \mathcal{E}'$ for some $t_0>0$ -- and hence for every $t\geq t_0$ -- or $\phi_t(z)\not\in \mathcal{E}'$ for every $t\geq 0$. A non-elliptic semigroup is said to be of \emph{finite shift} if for any orbit $z\mapsto\phi_t(z)$ there exists a horocycle at the Denjoy-Wolff point which does not intersect the given orbit. Clearly, if a semigroup in $\mathbb D$ is a finite shift semigroup, all  orbits  convergence tangentially to the Denjoy-Wolff point, but the converse is not true.

The notion of finite shift has been studied for iteration of holomorphic self-maps of the unit disc in \cite{PC, CDP} and in \cite[Section 17.7]{BCDbook} and \cite{Kar} for semigroups in $\mathbb D$.

All hyperbolic semigroups and  parabolic semigroups of zero hyperbolic step are of infinite shift (see for instance \cite[Prop. 17.7.3]{BCDbook}). Hence, only parabolic semigroups of \emph{positive} hyperbolic step may have finite shift.

Therefore, from now on in this introduction, $(\phi_t)$ denotes a parabolic semigroup in $\mathbb D$ of positive hyperbolic step  and $h$ its Koenigs function. Let  $\Omega:=h(\mathbb D)$ be the associated \emph{Koenigs domain}. As recalled before, in this case we may have, up to translations, that $
\Omega\subset \mathbb H$ or $
\Omega\subset -\mathbb H$, where $\mathbb{H}=\{z\in\mathbb{C}:\mathsf{Re}\,z>0\}$. The two cases are not equivalent in the sense of holomorphic models (see \cite[Remark 9.3.6]{BCDbook}), but for our purposes they are essentially the same, as only some change of signs should be considered. So we will assume that $\Omega\subset \mathbb H$ for the sake of simplicity.

By \cite[Theorem 17.7.6]{BCDbook}, if there exists $a>0$ such that $a+\mathbb{H}\subset\Omega\subset\mathbb{H}$ then $(\phi_t)$ is of finite shift, while, it is of infinite shift if $\Omega$ does not contain any vertical semi-sector, i.e. sets of the form $p_0+iV(\alpha,0))$ where $p_0\in\mathbb{H}$ and $V(\alpha,0)=\{z\in\mathbb{H}:-\alpha<\arg z<0\}$ for some $\alpha\in(0,\pi/2)$. These results suggest that $(\phi_t)$ is of finite shift if and only if $\Omega$ `asymptotically contains a half space'.

Karamanlis \cite{Kar} made precise the previous idea and gave a necessary and sufficient condition for $(\phi_t)$ to be of finite shift in terms of an integral that estimates `how much' larger and larger semi-sectors are contained in $\Omega$ (see Theorem~\ref{thm_karamanlis} for the precise statement).

In this paper, we provide other characterizations of finite shift. To this end, we need to introduce and recall some notations.

If $(\phi_t)$ is of finite shift, then $\Omega$ is `semi-conformal' at infinity (see Proposition \ref{prop: shift_reg}) and thus it has \emph{inner tangent at infinity along the positive real semi-axis}:\begin{equation}\label{eq:intro_cond}
	\text{for all $\beta\in(0,\pi/2)$ there exists $r_\beta>0$ such that $\Gamma(\beta,r_\beta)\subset\Omega$},
\end{equation} where $\Gamma(\beta,r_\beta):=\{z:\mathsf{Re}\,z>0,|\arg z|<\beta,|z|>r_\beta\}$. In particular the positive real semi-axis is eventually contained in $\Omega$.

In \cite{Br} (see also \cite[Chapter 16]{BCDbook}), three quantities called \emph{speeds of convergence} have been attached to non-elliptic semigroups. These are the relative hyperbolic distances between the origin, the point $\phi_t(0)$ and its hyperbolic projection to the diameter joining the origin to the Denjoy-Wolff point, seen as functions of the parameter $t\geq 0$. The one from the origin to the projection is the \emph{orthogonal speed} $v^O_\phi(t)$, while the distance of the projection from $\phi_t(0)$ is the \emph{tangential speed} $v^T_\phi(t)$.

\begin{theorem}\label{main} Let $(\phi_t)$ be a parabolic semigroup in $\mathbb D$ of positive hyperbolic step. Let $\tau\in\partial \mathbb D$ be its Denjoy-Wolff point, $h$ the associated Koenigs function and $\Omega=h(\mathbb D)$. Then the following are equivalent:
\begin{enumerate}[label=(\alph*)]
\item\label{t_sh} $(\phi_t)$ is of finite shift,
\item\label{t_sp} there exists $C>0$ such that for all $t\geq 1$,
 \[
	\left|v^O_\phi(t)-\frac{1}{2}\log t\right|+ \left|v^T_\phi(t)-\frac{1}{2}\log t\right|\leq C.
	\]
\item \label{t_hy} there exists $r_0>0$ such that $[r_0,+\infty)\subset \Omega$, the curve $[r_0,+\infty)\ni r\mapsto r\in \Omega$ is a $k_\Omega$-quasi-geodesic and
	\[
	\limsup\limits_{r\to+\infty}\,k_\Omega \left(r, h\left(\frac{r-1}{r+1}\tau\right)\right)=\limsup\limits_{r\to+\infty}\,k_\mathbb{\mathbb D} \left(h^{-1}(r), \frac{r-1}{r+1}\tau\right)<+\infty,
	\]
\item\label{t_ser} $\Omega$ satisfies $\eqref{eq:intro_cond}$ and, if $r_0>0$ is such that $[r_0,+\infty)\subset\Omega$
	\[
	\sum_{j\geq \lceil r_0 \rceil} \frac{1}{\inf\{y>0:j-i y\notin \Omega_\ast\}}<+\infty,
	\]
where $\Omega_\ast:=\{z\in\Omega:z+t\in\Omega\text{ for all }t\geq 0\}$.
\end{enumerate}
\end{theorem}

Some remarks about the meaning of the previous characterizations are in order. First, \ref{t_sp} says that $(\phi_t)$ is of finite shift if and only if its speed of convergence to $\tau$ is exactly like that of a parabolic group of automorphisms of $\mathbb D$. In \cite{Cor}  the author proved that the tangential speed goes at most like $\frac{1}{2}\log t$, while the orthogonal speed might go faster (see also \cite{BCK}).

Characterization \ref{t_hy} reflects the fact that the hyperbolic geometry of $\Omega$ around the real half-line should be eventually similar (in the Gromov sense) to that of $\mathbb H$ in the case of finite shift.

Finally, characterization \ref{t_ser} gives a (Euclidean) description of the shape of $\Omega$. The domain $\Omega_\ast$ is the `starlike-fication' of $\Omega$ (as defined in \cite{BR}) with respect to the real axis, and we prove that the associated semigroup is of finite shift if and only if $(\phi_t)$ is. For $\Omega_\ast$, Karamanlis' integral condition can be simplified in terms of a series which estimates how fast the height of the boundary of $\Omega_\ast$ decreases to $-\infty$.

We give an example of application of the previous theorem:
\begin{example}
 Let us denote by $b(j)$ the term $\inf\{y>0:j-i y\notin \Omega_\ast\}$ appearing in the series of the point \ref{t_ser}. For  $\epsilon\geq 0$ consider the simply connected domain
	\[
	\Omega_{\epsilon}:=\{z=x+iy\in\mathbb{H}:y>-x|\log x|^{1+\epsilon}\},
	\]
 which can be seen as the Koenigs domain of a parabolic semigroup $(\phi_t^\epsilon)$ in $\mathbb{D}$ with positive hyperbolic step.
	For $\epsilon=0$,  taking  $r_0=n_0=2$,  it is clear that for any $j\geq 2$ we have $b(j)=j\log j$. Since the series $\sum_{j=2}^{\infty}\frac{1}{j\log j}$ diverges, $(\phi_t^0)$ is of infinite shift.
	
	On the other hand, for $\epsilon>0$, still  taking $r_0=n_0=2$, the series $\sum_{j=2}^{\infty}1/b(j)$ compares with the improper integral
	\[
	\int_2^{\infty}\frac{\mathrm{d}x}{x(\log x)^{1+\epsilon}}=\int_{\log 2}^\infty\frac{\mathrm{d}u}{u^{1+\epsilon}}<+\infty,
	\] so it is convergent. Thus for any $\epsilon>0$ the shift of $(\phi_t^\epsilon)$ is finite.
\end{example}

The proof of the Theorem~\ref{main} relies on a mix of different techniques from complex analysis, metric spaces (in particular Gromov's shadowing lemma), hyperbolic geometry and harmonic analysis.

After all the preliminary notions needed throughout the article are presented in Section \ref{sec:pre}, the shift of a semigroup is introduced in Section \ref{sec:shift} together with some additional results. In the last sections we give the proof of Theorem~\ref{main}. 

\section{Preliminaries}\label{sec:pre}

\subsection{Hyperbolic metrics and distances}

Let $\mathbb{D}$ denote the open unit disc in the complex plane. Here we can define a \emph{hyperbolic norm}
\[
\varkappa_\mathbb{D}(z;v):=\frac{|v|}{1-|z|^2},\quad z\in\mathbb{D},\,v\in\mathbb{C}\cong T_z\mathbb{D}.
\]

Then we can assign a length to each piecewise $C^1$ path $\gamma:[0,1]\to\mathbb{D}$
\[
\ell_\mathbb{D}(\gamma)=\int_0^1\varkappa_\mathbb{D}(\gamma(s),\gamma'(s))\,\mathrm{d}s
\]
and we have the integrated distance
\[
k_\mathbb{D}(z,w):=\inf_{\gamma\in\Gamma(z,w)}\ell_\mathbb{D}(\gamma),\quad\text{for all } z,w\in\mathbb{D},
\]
where
\[
\Gamma(z,w):=\{\gamma:[0,1]\to\mathbb{D},\gamma\in C^1_{\mathrm{pw}}([0,1]),\gamma(0)=z,\gamma(1)=w\},
\]
which is called the \emph{hyperbolic distance} in $\mathbb{D}$ (or also \emph{Poincar\'e distance}). The following is an explicit expression for such distance:
\[
k_\mathbb{D}(z,w)=\frac{1}{2}\log\frac{1+\left|\frac{w-z}{1-\overline{w}z}\right|}{1-\left|\frac{w-z}{1-\overline{w}z}\right|}
\]
(see \cite[Theorem 1.3.5]{BCDbook}). It is a consequence of the Schwarz-Pick Lemma that the distance $k_\mathbb D$ is invariant for automorphisms of the unit disc.
 
For any simply connected domain $\Omega\subsetneq\mathbb{C}$ there exists a conformal equivalence $f:\mathbb{D}\to\Omega$ by the well-known Riemann Mapping Theorem. We can define a hyperbolic distance on $\Omega$ as
\[
k_\Omega(z,w):=k_\mathbb{D}(f^{-1}(z),f^{-1}(w)).
\]
The definition does not depend on the choice of $f$. From this definition one deduces \emph{conformal invariance} of hyperbolic metrics: if $F$ is a biholomorphism from $\Omega$ to $\Omega_1$, which are simply connected domains $\subsetneq \mathbb C$, then $k_{\Omega}(z,w)=k_{\Omega_1}(F(z),F(w))$ for any $z,w\in\Omega$. Indeed, if we have a map $f$ from $\mathbb D$ onto $\Omega$ as above, then $F\circ f$ is a conformal equivalence $\mathbb D\cong \Omega_1$, so both terms coincide by definition with $k_\mathbb{D}(f^{-1}(z),f^{-1}(w))$. 

A \emph{geodesic} is a smooth curve along which the hyperbolic length between two points correspond to their (hyperbolic) distance. For the unit disc equipped with $k_\mathbb{D}$, geodesic lines are given by diameters and arcs of circles intersecting $\partial\mathbb{D}$ orthogonally. There exist a unique geodesic between two points in $\mathbb{D}$, and a unique geodesic ray starting from a point in the disc and converging to a fixed point in the boundary.

For any simply connected domain $\Omega\subsetneq\mathbb{C}$, since any Riemann map $f:\mathbb{D}\to\Omega$ is an isometry when we consider the hyperbolic distances, we still have uniqueness of geodesic between two points of $\Omega$. For dealing with the boundary, one can consider the \emph{Carath\'eodory boundary} of the prime ends $\partial_C\Omega$. Hence, there exists a unique geodesic ray arising from a point of $\Omega$ and converging to a fixed prime end in the Carath\'eodory topology (see  \cite[Chapter 4]{BCDbook}).
\subsection{Quasi-geodesics}
A \emph{quasi-geodesic} in a simply connected domain $\Omega$ for its hyperbolic metric $k_\Omega$ is a Lipschitz curve $\sigma:[0,+\infty)\to\Omega$ such that $k_\Omega(\sigma(0),\sigma(t))\to+\infty$ for $t\to+\infty$ and there exist $A\geq 1$ and $B\geq 0$ for which
\[
\ell_\Omega(\sigma;s,t)\le Ak_\Omega(\sigma(s),\sigma(t))+B\quad\text{ for all }0\le s<t<+\infty.
\]
Here $\ell_\Omega(\sigma;s,t)$ denotes the length of the restricted curve $\sigma\mid_{[s,t]}$ with respect to the hyperbolic metric $k_\Omega$.
A quasi-geodesic satisfying this property for a fixed pair $(A,B)$ is called a \emph{$(A,B)$-quasi-geodesic}.

In general, it is hard to construct geodesic lines for the hypebolic metric of an arbitrary simply connected domain, as we may not have a complete description of the Riemann map and we cannot compute the metric exactly. Nonetheless, we may use some estimates on the metric in order to decide whether a given curve is a quasi-geodesic or not. The following results asserts that -- in some sense -- quasi-geodesics are close to geodesics:
\begin{proposition}[Shadowing Lemma]\label{prop:shad}
	Let $\Omega\subsetneq\mathbb{C}$ be a simply connected domain which is starlike at infinity. Let $A\geq 1$, $B\geq 0$. Then there exists $M=M(A,B,\Omega)>0$ such that for any $(A,B)$-quasi-geodesic $\sigma:[0,+\infty)\to\Omega$ there exist a $k_\Omega$-geodesic $\eta:[0,+\infty)\to\Omega$ such that $\sigma(0)=\eta(0)$ and for any $t\ge 0$ it is
	\[
	\inf_{s\geq 0} k_\Omega(\sigma(t),\eta(s))<M,\quad\inf_{s\ge0} k_\Omega(\eta(t),\sigma(s))<M.
	\]
	Furthermore, $\sigma$ and $\eta$ converge to the same prime end in the Carath\'eodory boundary $\partial_C\Omega$.
\end{proposition}
For a proof, see \cite[Theorem 6.3.8]{BCDbook}.

\subsection{Holomorphic semigroups in the unit disc}
A \emph{continuous semigroup of holomorphic self-maps in the unit disc} $(\phi_t)_{t\ge 0}$, or shortly a \emph{semigroup in $\mathbb{D}$}, is given by a family of holomorphic functions $\phi_t\in\mathsf{Hol}(\mathbb{D},\mathbb{D})$ such that the map $t\mapsto\phi_t$ from $\mathbb{R}^+:=[0,+\infty)$ to $\mathsf{Hol}(\mathbb{D},\mathbb{D})$:
\begin{enumerate}[label=(\roman*)]
	\item  it is a semigroup homomorphism from $(\mathbb{R}^+,+)$ to $(\mathsf{Hol}(\mathbb{D},\mathbb{D}),\circ)$;
	\item it is continuous by taking the Euclidean topology in $\mathbb{R}^+$ and the topology of uniform convergence on compact subsets in $\mathsf{Hol}(\mathbb{D},\mathbb{D})$.
\end{enumerate}

A semigroup as above is \emph{non-elliptic} when the maps $\phi_t$ have no fixed point in the unit disc: if so, there exists a unique point $\tau\in\partial\mathbb{D}$, called the \emph{Denjoy-Wolff point} of the semigroup, for which $\lim\limits_{t\to+\infty}\phi_t(z)=\tau$ for any $z\in\mathbb{D}$ and the convergence is uniform on any compact subset $K\subset\mathbb{D}$. By Denjoy-Wolff Theorem \cite[Theorem 1.8.4]{BCDbook} one has that for all $t>0$ and $R>0$
\[
\phi_t(\mathcal{E}(\tau,R))\subset\mathcal{E}(\tau,R):=\{z\in\mathbb{D}:|\tau-z|^2<R(1-|z|^2)\}
\]
The set $\mathcal{E}(\tau,R)$ defined above is called a \emph{horocycle} of (hyperbolic) radius $R$ and center $\tau$. In Euclidean terms it is the open disc contained inside $\mathbb{D}$ of radius $R/(R+1)$ and tangent at $\tau$ to $\partial\mathbb{D}$.

For any $t\geq 0$, the angular derivative at the Denjoy-Wolff point exists and moreover it is $\phi_t'(\tau)=e^{-\lambda t}$, where $\lambda\geq 0$: see \cite[Theorem 8.3.1]{BCDbook}. We say that a non-elliptic semigroup in $\mathbb{D}$ $(\phi_t)$ is \emph{parabolic} when $\lambda=0$, otherwise it is \emph{hyperbolic} of \emph{spectral value} $\lambda>0$.

The \emph{$1$-hyperbolic step} is  $s_1((\phi_t),z):=\lim\limits_{t\to+\infty}k_\mathbb{D}(\phi_t(z),\phi_{t+1}(z))$. For a parabolic semigroup, it can be positive for all $z\in\mathbb{D}$, otherwise it must be zero for all $z\in\mathbb{D}$. In the first case the semigroup is said to have \emph{positive hyperbolic step}, in the second one it has \emph{zero hyperbolic step}.

\subsection{Koenigs maps and domains}
For any non-elliptic semigroup $(\phi_t)$ in the unit disc one can associate a holomorphic map $h:\mathbb{D}\to\Omega\subset\mathbb{C}$ such that $\Omega$ is a simply connected domain which is \emph{starlike at infinity} in the positive direction of the imaginary axis, i.e. $\Omega+it\subset\Omega$ for all $t\geq 0$, and the relation $h(\phi_t(z))=h(z)+it$ holds for all $t\geq 0$ and $z\in\mathbb{D}$. Let $\Omega':=\bigcup\limits_{t\geq 0}(\Omega-it)$. Up to translations, it can be chosen in a unique way so that:
\begin{enumerate}[label=(\alph*)]
	\item $\Omega'$ coincides with the strip $\mathbb{S}_{0,\pi/\lambda}:=\{w\in\mathbb{C}: 0<\mathsf{Re}\,w<\pi/\lambda\}$ if $(\phi_t)$ is hyperbolic of spectral value $\lambda>0$.
	\item $\Omega'=\mathbb{H}:=\{w\in\mathbb{C}:\mathsf{Re}\,w>0\}$ or $\Omega'=-\mathbb{H}$ if $(\phi_t)$ is parabolic of positive hyperbolic step;
	\item $\Omega'=\mathbb{C}$ if $(\phi_t)$ is parabolic of zero hyperbolic step.
\end{enumerate}
This is proved, for instance, in \cite[Theorem 9.3.5]{BCDbook} and -- for the uniqueness part -- \cite[Proposition 9.3.10]{BCDbook}.
The function $h$ will be called the \emph{Koenigs map} associated to $(\phi_t)$. The set $\Omega=h(\mathbb{D})$ is called the \emph{Koenigs domain} associated to the semigroup. The uniqueness of the Koenigs map $h$ is meant up to translations.

Note that any simply connected domain $D$ in $\mathbb{C}$ with the starlike at infinity property as above can be seen as a Koenigs domain of a semigroup: just take a Riemann map $f$ from the unit disc to the domain $D$ and define\[
\phi_t^D(z):=f^{-1}(f(z)+it),\quad \text{for all }z\in\mathbb{D},\,t\geq 0.
\]

\subsection{Non-tangential convergence} For $R>1$ and $\tau\in\partial\mathbb{D}$, a \emph{Stolz region} of amplitude $R$ and of vertex $\tau$ is the set
\[
\mathcal{S}(\tau,R):=\{z\in\mathbb{D}:|\tau-z|<R(1-|z|)\}.
\]
We say that a sequence of points in $\mathbb{D}$ $z_n\to\tau$ converges to $\tau$ \emph{non-tangentially} if the $z_n$ are eventually contained in a Stolz region of vertex $\tau$. 
\begin{remark}[Notation]
We use the symbol $\angle\lim\limits_{z\to\tau}f(z)$ to denote the non-tangential limit of a function $f$ defined on $\mathbb{D}$. When it exists, $\angle\lim\limits_{z\to\tau}f(z)$ is the common limit of $f(z_n)$ for any sequence of points converging to $\tau$ non-tangentially.
\end{remark}
In the same fashion, we say that a curve $\gamma:[0,+\infty)\to\mathbb{D}$ converges to $\tau$ non-tangentially if $\gamma(t)\to\tau$ as $t\to+\infty$ and for $t$ enough $\gamma(t)$ belongs to $\mathcal{S}(\tau,R)$ for some $R>1$. This is equivalent to say that the cluster set of $\arg (1-\overline{\tau}\phi_t(z))$ for $t\to+\infty$ is contained in $(-\pi/2,\pi/2)$. On the other end, we say that the convergence is \emph{tangential} when this cluster set is equal to ${-\pi/2}$ or ${\pi/2}$.

In \cite{BCDGZ} it is proved the following result linking the convergence behavior of the orbits of a non-elliptic semigroup with the Euclidean shape of its Koenigs domain. Let $\Omega\subset \mathbb C$ be a domain starlike at infinity and let $p\in\Omega$. Let 
	\begin{align*}
		\begin{aligned}
			\widetilde{\delta^+_p}(t)&:=\inf\,\{|\zeta-p-it|:\zeta\in\partial\Omega,\mathsf{Re}\,\zeta\geq\mathsf{Re}\,p\},\\
			\widetilde{\delta^-_p}(t)&:=\inf\,\{|\zeta-p-it|:\zeta\in\partial\Omega,\mathsf{Re}\,\zeta\le\mathsf{Re}\,p\}.
		\end{aligned}
	\end{align*}
Note that $\widetilde{\delta^+_p}(t)$ is the distance of $p+it$ from the boundary of $\Omega$ which stays ``on the right'' of the line $\mathsf{Re}\,\zeta=\mathsf{Re}\,p$, and, similarly, $\widetilde{\delta^-_p}(t)$ is the distance ``on the left''. We set
\[
	\delta_p^+(t)=\min\{\widetilde{\delta^+_p}(t),t\},\quad	\delta_p^-(t)=\min\{\widetilde{\delta^-_p}(t),t\}.
	\]

\begin{theorem}[{\cite[Theorem 1.2 and 1.3]{BCDGZ}}]\label{thm:quasi-sym-nontg}
	Let $(\phi_t)$ be a non-elliptic semigroup in $\mathbb{D}$ with Denjoy-Wolff point $\tau\in\partial\mathbb{D}$. Let $h:\mathbb{D}\to \mathbb{C}$ be its Koenigs map and $\Omega=h(\mathbb{D})$ its Koenigs domain. 		
	\begin{enumerate}[label=(\alph*)]
		\item For some (and hence any) $z\in\mathbb{D}$ the orbit $t\mapsto\phi_t(z)$ converges \emph{non-tangentially} to $\tau$ if and only if for some (and hence any) $p\in\Omega$ there is a pair of constants $0<c<C$ such that for all $t\ge 0$ \[
		c\delta^+_p(t)\le\delta^-_p(t)\le C\delta^+_p(t).
		\]
		\item For some (and hence any) $z\in\mathbb{D}$ the orbit $t\mapsto\phi_t(z)$ converges \emph{tangentially} to $\tau$  if and only if one of the following holds:
		\begin{itemize}
			\item $\lim_{t\to+\infty}\delta^+_p(t)/\delta^-_p(t)=+\infty$ for some (and hence any) $p\in\Omega$;
			\item $\lim_{t\to+\infty}\delta^+_p(t)/\delta^-_p(t)=0$ for some (and hence any) $p\in\Omega$.
		\end{itemize}
		In the first case, the slope at any point in $\mathbb{D}$ is $\{-\pi/2\}$, while in the second one we have slope $\{\pi/2\}$ at all points of the unit disc.
	\end{enumerate}
\end{theorem}
In particular, hyperbolic semigroups have always orbits converging to the Denjoy-Wolff point non-tangentially.

\subsection{Speeds of convergence}
We conclude the preliminaries by introducing the \emph{speeds of convergence} of a (non-elliptic) semigroup. For more details we refer to the paper \cite{Br} and to Chapter 16 in \cite{BCDbook}.

For any semigroup $(\phi_t)$ in $\mathbb{D}$ with Denjoy-Wolff $\tau\in\partial\mathbb{D}$, the \emph{total speed} is the function of $t\geq 0$ given by the hyperbolic distance from the origin of $\phi_t(0)$:
\[
v_\phi(t):=k_\mathbb{D}(0,\phi_t(0)).
\]

If we denote by $\xi_t$ the unique point along the diameter from $-\tau$ to $\tau$ (a $k_\mathbb{D}$-geodesic line) which minimizes the hyperbolic distance from $\phi_t(0)$, then the \emph{orthogonal speed} of the semigroup at time $t\geq 0$ is
\[
v_\phi^O(t):=k_\mathbb{D}(0,\xi_t),
\]
while the \emph{tangential speed} is given by the distance
\[
v_\phi^T(t):=k_\mathbb{D}(\phi_t(0),\xi_t).
\]

An important fact is the following `hyperbolic Pythagoras' Theorem'.
\begin{proposition}[{\cite[Proposition 3.4]{Br}}]\label{prop:pyt}
Let $(\phi_t)$ be a non-elliptic semigroup in $\mathbb{D}$. Then for all $t\geq 0$ it is
\[
v_\phi^O(t)+v_\phi^T(t)-\frac{1}{2}\log 2\leq v_\phi(t)\leq v_\phi^O(t)+v_\phi^T(t).
\]
\end{proposition}
Moreover one can show that the tangential speed is bounded from above by the orthogonal speed up to a constant: $v^T_\phi(t)\leq v^O_\phi(t)+4\log 2$ (see \cite[Proposition 5.4]{Br} for a proof).

These quantities can be used to deduce some Euclidean properties of the semigroup: see \cite[Proposition 3.8]{Br}. Here we just recall the fact that the boundedness from above of the tangential speed corresponds to non-tangential convergence of the orbits of the semigroup to the Denjoy-Wolff point. In \cite{Cor} it is established the following asymptotic upper bound for the tangential speed:
\[
\limsup_{t\to+\infty}\left[v^T_\phi(t)-\frac{1}{2}\log t\right]<+\infty.
\]

\section{The shift of a non-elliptic semigroup}\label{sec:shift}

\begin{definition}
	
Let $(\phi_t)$ be a semigroup of holomorphic self-maps in $\mathbb{D}$ which is non-elliptic with Denjoy-Wolff point $\tau\in\partial\mathbb{D}$. It is of \emph{finite shift} if there exists $R>0$ and $z\in\mathbb{D}$ such that $\phi_t(z)\notin \mathscr{E}(\tau,R)$ for all $t\geq 0$. Otherwise we say that $(\phi_t)$ is of \emph{infinite shift}.

\end{definition}

If we translate the definition above into the right half-plane $\mathbb{H}$ via the Cayley transform $C_\tau:z\mapsto (\tau+z)/(\tau-z)$, we easily get that $(\phi_t)$ has finite shift if and only if there exist $R>0$ and $w\in\mathbb{H}$ such that $\mathsf{Re}\,\psi_t(w)<R$ for all $t\geq 0$, where $\psi_t=C_\tau\circ\phi_t\circ C_\tau^{-1}$.

The next Proposition shows that being of finite shift is a property which is independent of the choice of the initial point.

\begin{proposition}[{\cite[Lemma 17.7.4]{BCDbook}}]
	Let $(\phi_t)$ be a non-elliptic semigroup in $\mathbb{D}$ with Denjoy-Wolff point $\tau\in\partial\mathbb{D}$ and suppose it is of finite shift. Then for every $z\in\mathbb{D}$ there exists $R(z)>0$ so that for all $t\geq 0$ it is $\phi_t(z)\notin \mathscr{E}(\tau,R(z))$.
\end{proposition}

If the orbits of a semigroup $(\phi_t)$ converge non-tangentially to the Denjoy-Wolff point $\tau$, the shift is infinite since a Stolz region with vertex $\tau$ is eventually contained in any horocycle centered in $\tau$. Therefore all hyperbolic semigroups are of infinite shift. It remains to study the case of \emph{parabolic} semigroups.

\begin{proposition}[{\cite[Proposition 17.7.3]{BCDbook}}]
	Let $(\phi_t)$ be a parabolic semigroup in $\mathbb{D}$ with Denjoy-Wolff point $\tau\in\partial\mathbb{D}$ and zero hyperbolic step. Then $(\phi_t)$ is of infinite shift.
\end{proposition}
Thus only parabolic semigroups with positive hyperbolic step may be of finite shift.
A crucial fact in studying the shift of a parabolic semigroup is that a finite shift correspond to a regularity property of the Denjoy-Wolff point:
\begin{proposition}[{\cite[Theorem 17.7.5]{BCDbook}}]\label{prop: shift_reg}
	Let $(\phi_t)$ be a parabolic semigroup in $\mathbb{D}$ of positive hyperbolic step, with Denjoy-Wolff point $\tau\in\partial\mathbb{D}$ and Koenigs map $h:\mathbb{D}\to\Omega\subset\mathbb{H}$. Then the semigroup is of finite shift if and only if $\tau$ is a boundary regular fixed point of the holomorphic map $\tilde{h}:=C_\tau^{-1}\circ h:\mathbb{D}\to\mathbb{D}$, i.e. $\angle\lim\limits_{z\to\tau} \tilde{h}(z)=\tau$ and the angular derivative $\tilde{h}'(\tau)\in (0,+\infty)$. In the case $\Omega\subset -\mathbb{H}$, the same result holds by considering the map $\tilde{h}:=(-C_\tau)^{-1}\circ h$.
\end{proposition}

In particular, when the shift is finite the map $\tilde{h}$ is \emph{semi-conformal} at the boundary fixed point $\tau$, since
\[
\angle \lim_{z\to\tau} \arg \left(\frac{\tau-\tilde{h}(z)}{\tau-z}\right)=0.
\]

This gives the following property for Stolz regions (see \cite[Lemma 13.2.3]{BCDbook}): for any $\epsilon>0$, $M>1$ and $M'>M$ there exists $\delta>0$ for which
\begin{equation*}
	\mathcal{S}(\tau,M)\cap D(\tau,\delta)\subset \tilde{h}(\mathcal{S}(\tau,M')\cap D(\tau,\epsilon)),
\end{equation*}
or equivalently
\begin{equation} \label{eq: stolz-finite}
	C_\tau(\mathcal{S}(\tau,M)\cap D(\tau,\delta))\subset h(\mathcal{S}(\tau,M')\cap D(\tau,\epsilon)).
\end{equation}

The right-hand side of \eqref{eq: stolz-finite} is contained in $\Omega$, hence \eqref{eq: stolz-finite} gives a prescription on the shape of $\Omega$ when the shift of the semigroup is finite. If we define $\Gamma(\alpha,r):=\{z\in\mathbb{H}:|z|>r,\,|\arg z|<\alpha\}$, it is easy to show that this condition is equivalent to the following:

\begin{property}\label{prty_cone}
		For all $\beta\in(0,\pi/2)$ there exists $r_\beta>0$ such that $\Gamma(\beta,r_\beta)\subset\Omega$.	
\end{property}
In this case we also say that $\Omega$ has an \emph{inner tangent at infinity along the positive real semi-axis}.
This property of the Koenigs domain is a \emph{necessary} hypothesis in order to have a finite shift. When it holds, there exists $r_0>0$ such that $[r_0,+\infty)\subset\Omega$. Moreover we can assume that $r_0+W_0\subset\Omega$, where $W_0=V(\beta_0,\beta_0)=\{\zeta\in\mathbb{C}:-\beta_0<\arg\zeta<\beta_0\}$ for some $\beta_0\in(0,\pi/2)$. Then it can be shown that the half-line $[2r_0,+\infty)$  is (the image of) a quasi-geodesic for the hyperbolic metric.

\begin{proposition}\label{prop: quasig}
	Let $\Omega$ be a simply connected domain in $\mathbb{C}$ satisfying Property \ref{prty_cone}. Choose $r_0$ as above. Then the curve $\sigma:[2r_0,+\infty)\to\Omega,\sigma(r)=r$ is a quasi-geodesic for the hyperbolic metric on $\Omega$.
\end{proposition}

\begin{proof}
	For $s_2>s_1\geq r_0$ one has \[
	k_\Omega(s_1,s_2)\geq k_\mathbb{H}(s_1,s_2)=\frac{1}{2}\log\frac{s_2}{s_1}.
	\]
	The first inequality follows from the monotonicity relation of hyperbolic metrics: as $\Omega\subset\mathbb H$, then $k_\Omega\geq k_\mathbb H$. It is a particular case of \cite[Proposition 1.3.10]{BCDbook}. The last follows from an easy computation by considering the conformal equivalence from $\mathbb D$ to $\mathbb H$ given by a Cayley transform.
	
	Again by the monotonicity result cited above one can show that $\Omega\supset r_0+W_0$ implies that $\ell_\Omega(\sigma;s_1,s_2)\leq \ell_{r_0+W_0}(\sigma;s_1,s_2)$. By symmetry  (see \cite[Proposition 6.1.3]{BCDbook}), $\sigma$ is a geodesic line for the metric $k_{r_0+W_0}$; thus, since the map $z\mapsto z-r_0$ maps $r_0+W_0$ conformally onto $W_0$ and we have conformal equivalence of hyperbolic metrics (see the beginning of Section \ref{sec:pre}), \[
	\ell_{r_0+W_0}(\sigma;s_1,s_2)=k_{r_0+W_0}(s_1,s_2)=k_{W_0}(s_1-r_0,s_2-r_0).
	\]
	So if $s_1\geq 2r_0$, then
\[
	\ell_\Omega(\sigma;s_1,s_2)\leq k_{W_0}(s_1-r_0,s_2-r_0)=\frac{\pi}{4\beta_0} \log \frac{s_2-r_0}{s_1-r_0}\leq \frac{\pi}{4\beta_0} \log \frac{s_2}{s_1}+\frac{\pi}{4\beta_0} \log 2;\]
	the second equality comes from the conformal equivalence $W_0\cong \mathbb{H}$ given by the map $w\mapsto w^{\pi/(2\beta_0)}$, while for the latter inequality the assumption $s_1\geq 2r_0$ implies that $s_1-r_0\geq s_1/2$. Thus $\sigma$ is a $(A,B)$-quasi-geodesic for $k_\Omega$ if we take $A=\pi/(2\beta_0)>1$ and $B=(\pi\log 2)/(4\beta_0)>0$.
\end{proof}

Another consequence of Proposition \ref{prop: shift_reg} is the following.

\begin{corollary}\label{cor: shift_metric}
	Let $(\phi_t)$ be a parabolic semigroup in $\mathbb{D}$ of positive hyperbolic step, with Denjoy-Wolff point $\tau\in\partial\mathbb{D}$ and Koenigs map $h:\mathbb{D}\to\Omega\subset\mathbb{H}$. The following are equivalent:
	\begin{enumerate}[label=(\roman*)]
		\item \label{cor_sh_1}$(\phi_t)$ is of finite shift.
		\item \label{cor_sh_2}$	\liminf\limits_{w\to\tau}\left[ k_\mathbb{D}(0,w)-k_\mathbb{D}(0,C_\tau^{-1}(h(w)))\right]<+\infty$.
		\item \label{cor_sh_3}There exists the non-tangential limit $\angle\lim\limits_{w\to\tau}\left[ k_\mathbb{D}(0,w)-k_\mathbb{D}(0,C_\tau^{-1}(h(w)))\right]=L\in\mathbb{R}$.
	\end{enumerate}

\end{corollary}

\begin{proof}
The fact that $\tau$ is a boundary regular  fixed point for the map $\tilde{h}=C_\tau^{-1}\circ h$ is equivalent to $\ref{cor_sh_1}$ by Proposition \ref{prop: shift_reg}. Hence,   by the Julia-Wolff-Carath\'eodory's Theorem \cite[Theorem 1.7.3]{BCDbook} this is equivalent  to $\alpha_{\tilde{h}}(\tau)<+\infty$, where $\alpha_{\tilde h}(\tau)$ is the dilation coefficient 
\[
\alpha_{\tilde{h}}(\tau)=\liminf\limits_{w\to\tau}\,\frac{1-|\tilde{h}(w)|}{1-|w|}.\]
Now by \cite[Lemma 1.4.5]{BCDbook}, 
\begin{equation}\label{eq:dilation_log}
\frac{1}{2}\log \alpha_{\tilde{h}}(\tau)=\liminf_{w\to\tau}\left[ k_\mathbb{D}(0,w)-k_\mathbb{D}(0,\tilde{h}(w))\right],
\end{equation}
hence $\ref{cor_sh_1}$ is equivalent to $\ref{cor_sh_2}$. Note that \eqref{eq:dilation_log} implies that $\alpha_{\tilde h}(\tau)$ is always greater than zero, as the term in brackets is bounded from below by $k_\mathbb{D}(\tilde h(0),\tilde h(w))-k_\mathbb{D}(0,\tilde{h}(w))\geq -k_\mathbb D(0,\tilde h(0))$, using Schwarz-Pick Lemma and triangle inequality.

The proof of \eqref{eq:dilation_log} relies on the fact that, by direct computation,
\begin{equation}\label{eq:dilation_pf}
k_\mathbb{D}(0,w)-k_\mathbb{D}(0,\tilde{h}(w))=\frac{1}{2}\log\left(
\frac{1-|\tilde h(w)|}{1-|w|}\right)+O(1)\quad\text{as }w\to\tau.
\end{equation}
Assuming $\ref{cor_sh_1}$, then Proposition \ref{prop: shift_reg} implies that $\angle\lim\limits_{w\to\tau}\tilde h(w)=\tau$, so by \cite[Proposition 1.7.4]{BCDbook}, \[
\angle\lim_{w\to\tau}\frac{1-|\tilde{h}(w)|}{1-|w|}=\alpha_{\tilde h}(\tau)>0.
\]
Combining this with \eqref{eq:dilation_pf}, there exists $\angle\lim\limits_{w\to\tau}\,[k_\mathbb{D}(0,w)-k_\mathbb{D}(0,\tilde{h}(w))]>-\infty$. Since $\ref{cor_sh_2}$ is equivalent to $\ref{cor_sh_1}$, this limit cannot be $+\infty$, so $\ref{cor_sh_3}$ holds.

As $\ref{cor_sh_3}\Rightarrow\ref{cor_sh_2}$ is trivial, we have done.
\end{proof}

\section{Proof of Theorem \ref{main}}\label{sec:proof}
We are going to prove independently the equivalences $\ref{t_sh}\Leftrightarrow\ref{t_sp}$, $\ref{t_sh}\Leftrightarrow\ref{t_hy}$ and then $\ref{t_sh}\Leftrightarrow\ref{t_ser}$, after some needed remarks.
\begin{proof}[Proof of Theorem \ref{main}, $\ref{t_sh}\Leftrightarrow\ref{t_sp}$]
	Consider the semigroup $(\psi_t)$ associated to $(\phi_t)$ in $\mathbb{H}$, $\psi_t=C_\tau\circ\phi_t\circ C_\tau^{-1}$ and write $\psi_t(1)=\rho_t e^{i\theta_t}$. Then $(\phi_t)$ is of finite shift if and only if there exists $M>0$ such that $\mathsf{Re}\,\psi_t(1)=\rho_t \cos \theta_t\in [1,M)$ for all $t\geq 0$. Equivalently
	\[
	0\leq \frac{1}{2}\log \rho_t - \frac{1}{2}\log\frac{1}{\cos \theta_t}<\frac{1}{2}\log M.
	\]
	Now $v^O_\phi(t)=(1/2)\log \rho_t$ and $v^T_\phi(t)-(1/2)\log(1/\cos \theta_t)=O(1)$ as $t\to+\infty$. See for instance equations (3.6) in \cite{Cor}. So $(\phi_t)$ is of finite shift if and only if $v^O_\phi(t)-v^T_\phi(t)=O(1)$ for $t\to+\infty$.
	
	To conclude the proof it is enough to show that for a semigroup of finite shift one has
	\[
	-\infty<\liminf_{t\to+\infty}\left[v^T_\phi(t)-\frac{1}{2}\log t\right]\leq\limsup_{t\to+\infty}\left[v^T_\phi(t)-\frac{1}{2}\log t\right]<+\infty.
	\]
	
	The last inequality is given in \cite[Theorem 1.1]{Cor}. For the first one, let assume that the Koenigs domain $\Omega$ has smooth boundary. This can be done by the `horocycle reduction' described in \cite[Proposition 3.7]{Cor}. In the case of finite shift, $\Omega$ must contain a vertical semi-sector $p+iV(\alpha,0)$ where $p\in\mathbb{H}$ and $V(\alpha,0)=\{\zeta\in\mathbb{C}:-\alpha<\arg\zeta<0\}$, as shown in \cite[Theorem 17.7.6(a)]{BCDbook}. Let us consider \[ \Omega^\sharp:=\Omega\cup\{z\in\mathbb{C}: 0<\mathsf{Re}\,z<\mathsf{Re}\,p,\,\mathsf{Im}\,z>\mathsf{Im}\,p\}.\] This is again a Koenigs domain of a parabolic semigroup $(\phi^\sharp_t)$. We can assume that its Denjoy-Wolff point is the same as $(\phi_t)$.
	
	It turns out that
	$\liminf\limits_{t\to+\infty}\,[v^T_\phi(t)-v^T_{\phi^\sharp}(t)]>-\infty$. We give here a sketch of the proof of this claim. Let $h$ and $h^\sharp$ be the respective Koenigs maps. It is possible to construct as in \cite[Section 4]{BCDGZ} two curves $\sigma:[0,+\infty)\to\Omega$ and $\sigma^\sharp:[0,+\infty)\to \Omega^\sharp$ such that $\lim\limits_{s\to+\infty}\mathsf{Im}\,\sigma(s)=\lim\limits_{s\to+\infty}\mathsf{Im}\,\sigma^\sharp(s)=+\infty$, $\lim\limits_{s\to+\infty}h^{-1}(\sigma(s))=\lim\limits_{s\to+\infty}(h^\sharp)^{-1}(\sigma^\sharp(s))=\tau$, which are quasi-geodesics for the metrics $k_\Omega$ and $k_{\Omega^\sharp}$ respectively and so that $\mathsf{Re}\,\sigma(s)\geq\mathsf{Re}\,\sigma^\sharp(s)\geq \mathsf{Re}\,p$ for any $s\geq s_0>0$. By possibly modifying the curves between $0$ and $s_0$, we may assume that this relation holds in general and that the initial points are the same. If we take the point $p+it$ for some $t\geq 0$ and $s_t$ such that $k_\Omega(p+it,\sigma(s_t))=\inf_s\,k_\Omega(p+it,\sigma(s))$, the $k_\Omega$-geodesic segment $L$ joining $p+it$ and $\sigma(s_t)$ must intersect the image of $\sigma^\sharp$ at some point $q\in L\cap\sigma([0,+\infty))$: the argument is analogous to the one shown in the proof of \cite[Theorem 4.1]{Cor} and uses Jordan's Curve Theorem. Hence
	\[
	\inf_s\,k_\Omega(p+it,\sigma(s))=k_\Omega(p+it,\sigma(s_t))=\ell_\Omega(L)\geq k_\Omega(p+it,q),
	\]
	since $q$ is an intermediate point along the geodesic $L$. Thus, recalling that $\Omega\subset\Omega^\sharp$ implies $k_\Omega\geq k_{\Omega^\sharp}$,
	\[\inf_s\,k_\Omega(p+it,\sigma(s))\geq k_\Omega(p+it,q)\geq k_{\Omega^\sharp}(p+it,q)\geq\inf_s\, k_{\Omega^\sharp}(p+it,\sigma^\sharp(s)).\]
	By \cite[Proposition 3.6]{Cor}, the two sides of the above inequality compare with $v^T_\phi(t)$ and $v^T_{\phi^\sharp}(t)$ respectively, up to additive constants not depending on $t$. This proves the claim.
	
	On the other hand, since $\Omega^\sharp\supset \mathsf{Re}\,p +i V(\alpha,0)$, by \cite[Theorem 4.2]{Cor} and \cite[Corollary  16.2.6]{BCDbook} it is also true that $\liminf\limits_{t\to+\infty}\,[v^T_{\phi^\sharp}(t)-(1/2)\log t]>-\infty$ and this concludes the proof.
\end{proof}

\begin{proof}[Proof of Theorem \ref{main}, $\ref{t_sh}\Leftrightarrow\ref{t_hy}$]
By Corollary \ref{cor: shift_metric}, using the conformal invariance of hyperbolic metric, the shift of $(\phi_t)$ is finite exactly when \begin{equation}\label{eq: cor_sh_rest}
\liminf_{w\to\tau}\left[ k_\Omega(h(0),h(w))-k_\mathbb{H}(1,h(w))\right]<+\infty\iff\angle\lim_{w\to\tau}\left[ k_\Omega(h(0),h(w))-k_\mathbb{H}(1,h(w))\right]\in \mathbb{R}.
\end{equation}
By assumption the half-line $[r_0,+\infty)$ is the image of a quasi-geodesic $\sigma$ for $k_\Omega$, which converges to the point at infinity in the Riemann sphere $\mathbb{C}_\infty$. This last is true also for the vertical half-line $[0,+\infty)\ni t\mapsto r_0+it\in \Omega$. As $h^{-1}(r_0+it)\to\tau$ as $t\to+\infty$ and \[
\{z\in\mathbb{C}:\mathsf{Re}\,z>r_0,\,\mathsf{Im}\,z>0\}\subset\Omega,
\]
one deduces by \cite[Proposition 3.3.5]{BCDbook} that $h^{-1}(r)\to\tau$ as $r$ goes to $+\infty$.
A $k_\Omega$-geodesic ray is provided by $\gamma:[1,+\infty)\to\Omega$, $\gamma(r):=h\circ C_\tau^{-1}(r)=h(\tau(r-1)/(r+1))$, as $[1,+\infty)$ is a geodesic for the metric $k_\mathbb{H}$ and $h\circ C_\tau^{-1}$ maps conformally $\mathbb{H}$ onto $\Omega$. Note that $\gamma(1)=h(0)$. We can extend the curve $\sigma$ with an arc in $\Omega$ joining $r_0$ to $h(0)$, still having a quasi-geodesic ray which now starts from the same point of $\gamma$.
By Proposition \ref{prop:shad}, there exists a constant $M>0$ such that for all $r\geq r_0$ it is $k_\Omega(r,\gamma)=k_\Omega(r,\gamma(s(r)))<M$, where $\gamma(s(r))$ is the (unique) hyperbolic projection of $r$ onto the geodesic $\gamma$, and the convergence $h^{-1}(r)\to\tau$ is non-tangential: hence by replacing $w$ with $h^{-1}(r)$ in \eqref{eq: cor_sh_rest}, the shift of $(\phi_t)$ is finite if and only if 
\begin{equation}\label{eq:dif4shift}
k_\Omega(h(0),r)-k_\mathbb{H}(1,r)=k_\Omega(h(0),r)-\frac{1}{2}\log r
\end{equation}
goes to a finite limit when $r\to+\infty$. By Proposition \ref{prop:pyt}, this implies that the quantity
\[
k_\Omega(h(0),r)-k_\Omega(h(0),(h\circ C_\tau^{-1})(s(r)))=k_\Omega(h(0),r)-k_\mathbb{H}(1,s(r))=k_\Omega(h(0),r)-\frac{1}{2}\log s(r)
\]
is bounded for $r\in[r_0,+\infty)$. Combining this with \eqref{eq:dif4shift}, we get that finite shift corresponds to boundedness from above of \[
\frac{1}{2}\left|\log \frac{s(r)}{r}\right|=k_\mathbb{H}(r,s(r))=k_\Omega(\gamma(r),\gamma(s(r)))
\] as $r_0<r\to+\infty$. But since $k_\Omega(r,\gamma(s(r)))<M$, this is equivalent to \[
\limsup_{r\to+\infty}\, k_\Omega(r,\gamma(r))<+\infty,\quad\text{where }\gamma(r)=h\circ C_\tau^{-1}(r)
\]
by using again Proposition \ref{prop:pyt}.
\end{proof}

It remains to prove the equivalence $\ref{t_sh}\Leftrightarrow\ref{t_ser}$.
In \cite{Kar} it is shown the following result.

\begin{theorem}\label{thm_karamanlis}
	Let $\Omega$ be the Koenigs domain of a parabolic semigroup in $\mathbb{D}$ and let assume that $\mathbb{H}$ is the smallest vertical half-plane containing $\Omega$. Assume that $\Omega$ satisfies Property \ref{prty_cone} and let $r_0$ such that $[r_0,+\infty)\subset\Omega$. Let denote by $\eta_\Omega(r)$ the Euclidean length divided by $r$ of the connected component of $\Omega\cap\{|z|=r\}$ containing the point $r\geq r_0$ in the real axis. Then the semigroup is of finite shift if and only if
	\begin{equation}\label{int1}
		\int_{r_0}^{\infty}\left(\frac{1}{\eta_\Omega(r)}-\frac{1}{\pi}\right)\,\frac{\mathrm{d}r}{r}<+\infty.
	\end{equation}
\end{theorem}

\begin{remark}\label{rem-kar}
	Since $\eta_\Omega(r)\in(\pi/2,\pi]$ for $r\geq r_0$, the integral \eqref{int1} is comparable with
	\begin{equation}\label{int2}
		\int_{r_0}^{\infty}\frac{\pi-\eta_\Omega(r)}{r}\,\mathrm{d}r
	\end{equation}
\end{remark}

In the proof of Theorem \ref{thm_karamanlis} developed in \cite{Kar} it is shown, by means of estimates of harmonic measures, that if $\Omega$ is a domain as above and, for $M>0$, $\Omega_M$ is the maximal domain contained in $\Omega$ whose boundary is of the form ${g(y)+iy,\,y\in\mathbb{R}}$ where $g\in\mathrm{Lip}_M(\mathbb{R},\mathbb{R}^+)$, then 
\begin{equation*}
	\int_{r_0}^{\infty}\frac{\pi-\eta_\Omega(r)}{r}\,\mathrm{d}r<+\infty\implies \int_{r_0}^{\infty}\frac{\pi-\eta_{\Omega_M}(r)}{r}\,\mathrm{d}r<+\infty,
\end{equation*}
assuming that $r_0\in\Omega_M$ without loss of generality; note that the converse is trivial since $\eta_{\Omega_M}\leq \eta_\Omega$. The same argument holds when $\Omega_M$ is replaced by the domain $\Omega_\ast$ without any significant change. $\Omega_\ast$ is the maximal domain contained in $\Omega$ which is starlike at infinity \emph{in the positive direction of the real axis}. It is easy to show that it has all the properties of a Koenigs domain. Then one concludes the following:

\begin{proposition}\label{prop-ast}
	Let $\Omega$ be as in Theorem \ref{main}. Then $\Omega$ is associated to a semigroup of finite shift if and only if the same is true for $\Omega_\ast$.
\end{proposition}
Let consider a strictly increasing sequence of real numbers $\{a_k\}_{k\geq 0}$ so that $a_0=0$ and $\lim\limits_{k\to +\infty}a_k=+\infty$, and another non-decreasing sequence $\{b_k\}_{k\geq 1}$ of positive real numbers such that $b_k\to+\infty$ and $a_k/b_k\to 0$ as $k\to+\infty$. Let $S_k:=\{z\in\mathbb{H}:a_{k-1}\leq \mathsf{Re}\,z \leq a_k,\,\mathsf{Im}\,z\leq -b_k\}$ for any $k\in\mathbb{N}$ and consider the `step' domain in the right half-plane
\[
\Sigma=\Sigma(a_k,b_k):=\mathbb{H}\setminus\bigcup\limits_{k=1}^{\infty}S_k,
\]
which is a simply connected domain and moreover it is starlike at infinity in the positive direction of the imaginary axis: see Figure \ref{step_domain}.
We can therefore construct a parabolic semigroup $(\phi_t^\Sigma)$ in $\mathbb{D}$ with positive hyperbolic step and having $\Sigma$ as an associated Koenigs domain.
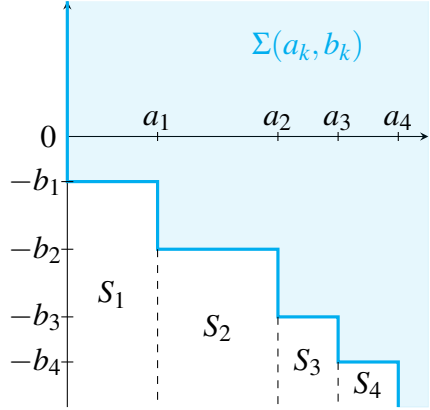
\begin{figure}[htb]
	\centering
	\begin{tikzpicture}[xscale=0.8, yscale=0.6]
		\draw [-stealth] (0,-6)--(0,3);
		\draw [-stealth] (0,0)--(6,0);
		\draw [very thick,cyan] (0,3) -- (0,-1) -- (1.5,-1) -- (1.5,-2.5) -- (3.5,-2.5) -- (3.5,-4) -- (4.5,-4) -- (4.5,-5) -- (5.5,-5) -- (5.5,-6);
		\draw [cyan,fill, opacity=0.1] (0,3) -- (0,-1) -- (1.5,-1) -- (1.5,-2.5) -- (3.5,-2.5) -- (3.5,-4) -- (4.5,-4) -- (4.5,-5) -- (5.5,-5) -- (5.5,-6)--(6,-6)--(6,3);
		\node [anchor=east] at (0,0) {$0$};
		\draw (1.5,0.1) -- (1.5,-0.1) node[anchor=south] {$a_1$};
		\draw (3.5,0.1) -- (3.5,-0.1) node[anchor=south] {$a_2$};
		\draw (4.5,0.1) -- (4.5,-0.1) node[anchor=south] {$a_3$};
		\draw (5.5,0.1) -- (5.5,-0.1) node[anchor=south] {$a_4$};
		\draw (-0.1,-1) -- (0.1,-1) node[anchor=east] {$-b_1$};
		\draw (-0.1,-2.5) -- (0.1,-2.5) node[anchor=east] {$-b_2$};
		\draw (-0.1,-4) -- (0.1,-4) node[anchor=east] {$-b_3$};
		\draw (-0.1,-5) -- (0.1,-5) node[anchor=east] {$-b_4$};
		\node at (4,2) {\textcolor{cyan}{$\Sigma(a_k,b_k)$}};
		\draw [dashed] (1.5,-2.5)--(1.5,-6);
		\draw [dashed] (3.5,-4)--(3.5,-6);
		\draw [dashed] (4.5,-5)--(4.5,-6);
		\node at (0.75,-3.5) {$S_1$};
		\node at (2.5,-4.25) {$S_2$};
		\node at (4,-5) {$S_3$};
		\node at (5,-5.5) {$S_4$};
	\end{tikzpicture}
	\caption{The domain $\Sigma(a_k,b_k)$.}
	\label{step_domain}
\end{figure}
\begin{question}\label{quest-step}
	Which properties of the two sequences $\{a_k\},\{b_k\}$ characterize the shift of $(\phi_t^\Sigma)$?
\end{question}

By the assumption $a_k/b_k\to0$, $\Sigma$ satisfies the hypothes of Theorem \ref{thm_karamanlis}, so by Remark \ref{rem-kar} the problem reduces to the study of the convergence of \eqref{int2} for $\Omega=\Sigma$. Let  $c_k:=|a_{k-1}-i b_k|$ and $d_k:=|a_k-i b_k|$ for any $k\in\mathbb{N}$. We may choose $r_0=c_1$ for instance.

Suppose that $r\in[c_k,d_k]$. Then
\[
\frac{\pi-\eta_\Sigma(r)}{r}=\frac{1}{r}\,\arctan \frac{\sqrt{r^2-b_k^2}}{b_k}.
\]
Note that $\sqrt{r^2-b_k^2}/b_k\leq a_{k-1}/b_k\to 0$. So for $k$ big enough it is $0<\sqrt{r^2-b_k^2}/b_k<1$. Recalling that for $x\in(0,1)$ it is $x>\arctan x > (\pi/4)x$, it is
\[
\frac{\pi}{4}\,\frac{a_{k-1}}{b_k\,c_k}\leq\frac{\pi-\eta_\Sigma(r)}{r}\leq\frac{a_k}{b_k\,d_k}
\] for any $r\in [c_k,d_k]$, $k\gg 1$. By our assumptions it is $c_k\sim b_k$, then for possibly larger $k$ it is also $2b_k > c_k$. Hence there exists $k_1\in \mathbb{N}$ such that for any index $k\geq k_1$ it is
\[
\int_{c_k}^{d_k}\frac{\pi-\eta_\Sigma(r)}{r}\,\mathrm{d}r\leq 2(d_k-c_k)\frac{a_k}{c_k\,d_k}=2\left(\frac{1}{c_k}-\frac{1}{d_k}\right)a_k
\]
and
\[
\int_{c_k}^{d_k}\frac{\pi-\eta_\Sigma(r)}{r}\,\mathrm{d}r\geq \frac{\pi}{4}(d_k-c_k)\frac{a_k}{c_k\,d_k}=\frac{\pi}{4}\left(\frac{1}{c_k}-\frac{1}{d_k}\right)a_{k-1}.
\]

Now let $r\in[d_k,c_{k+1}]$. Here
\[
\frac{\pi-\eta_\Sigma(r)}{r}=\frac{1}{r}\,\arctan \frac{a_k}{\sqrt{r^2-a_k^2}}.
\]
Note that $a_k/\sqrt{r^2-a_k^2}\leq a_k/b_k\to 0$. Moreover $\sqrt{r^2-a_k^2}\geq r-a_k$. For $k$ sufficiently big $b_k\geq \sqrt{3}a_k$, hence $r-a_k\geq d_k-a_k\geq a_k$ implying that $r-a_k\geq r/2$. We deduce the following inequalities as $k\geq k_2$ for $k_2\in\mathbb{N}$ big enough:
\[
\frac{\pi}{4}\frac{a_k}{r^2}\leq\frac{\pi-\eta_\Sigma(r)}{r}\leq 2\,\frac{a_k}{r^2}\qquad(r\in[d_k,c_{k+1}]),
\]
\[
\int_{d_k}^{c_{k+1}}\frac{\pi-\eta_\Sigma(r)}{r}\,\mathrm{d}r\leq 2a_k\int_{d_k}^{c_{k+1}}\frac{\mathrm{d}r}{r^2}=2\left(\frac{1}{d_k}-\frac{1}{c_{k+1}}\right)a_k,
\]
\[
\int_{d_k}^{c_{k+1}}\frac{\pi-\eta_\Sigma(r)}{r}\,\mathrm{d}r\geq \frac{\pi}{4}a_k\int_{d_k}^{c_{k+1}}\frac{\mathrm{d}r}{r^2}=\frac{\pi}{4}\left(\frac{1}{d_k}-\frac{1}{c_{k+1}}\right)a_k.
\]

Let $k_0=\max\{k_1,k_2\}$. From all these inequalities, the integral \eqref{int2} is smaller than a constant plus
\[
\int_{c_{k_0}}^\infty\frac{\pi-\eta_\Sigma(r)}{r}\,\mathrm{d}r\leq 2\sum_{k=k_0}^{\infty}\left(\frac{1}{c_k}-\frac{1}{c_{k+1}}\right)a_k=\frac{2a_{k_0}}{c_{k_0}}+2\sum_{k=k_0+1}^{\infty}\frac{a_k-a_{k-1}}{c_k},
\]
while it can be estimated from below by
\[
\int_{d_{k_0}}^\infty\frac{\pi-\eta_\Sigma(r)}{r}\,\mathrm{d}r\geq \frac{\pi}{4}\sum_{k=k_0}^{\infty}\left(\frac{1}{d_k}-\frac{1}{d_{k+1}}\right)a_k=\frac{\pi}{4}\,\frac{a_{k_0}}{d_{k_0}}+\frac{\pi}{4}\sum_{k=k_0+1}^{\infty}\frac{a_k-a_{k-1}}{d_k}.
\]

But since $\lim\limits_{k\to +\infty}c_k/b_k=\lim\limits_{k\to +\infty}d_k/b_k=1$, we have the following conclusion, giving an answer to Question \ref{quest-step}:
\begin{proposition}
	The semigroup $(\phi_t^\Sigma)$ is of finite shift if and only if
	\begin{equation}\label{series-step}
		\sum_{k=1}^{\infty} \frac{a_k-a_{k-1}}{b_k}<+\infty.
	\end{equation}
\end{proposition}
Going back to the domain $\Omega$, let assume that the Property \ref{prty_cone} holds for it. By Proposition \ref{prop-ast}, it is associated to a semigroup of finite shift if and only if $\Omega_\ast$ does. Take $a_0=0$, $a_k=n_0-1+k$ and $b_k=b(a_k)$ for $k\geq 1$. Then the step domain $\Sigma=\Sigma(a_k,b_k)$ must contain $\Omega_\ast$ as this last one is starlike at infinity in the positive direction of both real and imaginary axis (see Figure \ref{step_omega}): hence $\pi-\eta_\Sigma(r)\leq \pi-\eta_{\Omega_\ast}(r)$ and
\[
\int_{n_0}^\infty\frac{\pi-\eta_{\Omega_\ast}(r)}{r}\,\mathrm{d}r\geq\int_{n_0}^\infty\frac{\pi-\eta_{\Sigma}(r)}{r}\,\mathrm{d}r.
\]
The last integral compares with the series given in \eqref{series-step} which in this case is
\[
\frac{n_0}{b(n_0)}+\sum_{k=2}^{\infty}\frac{1}{b(n_0+k-1)}
\]
and this last one compares with $\sum_{j= n_0}^{\infty}1/b(j)$. Hence when $\sum_{j= n_0}^{\infty}1/b(j)=+\infty$, $\Omega_\ast$ is associated to a semigroup of infinite shift. On the other side consider the other `step' domain $\Sigma'=\Sigma(a'_k,b'_k)$ where \[
a'_0=0,\quad a'_k=a_{k+1},\quad b'_k=b_k\quad\text{for }k\in\mathbb{N}.
\] Then $\Sigma'\cap\{\mathsf{Re}\,z>n_0\}\subset \Omega_\ast\cap \{\mathsf{Re}\,z>n_0\}$, as shown in Figure \ref{step_omega}. If $r\geq 2n_0$, then $\eta_{\Sigma'}(r)-\eta_{\Omega_\ast}(r)$ is bounded from above by
\[
\arctan \frac{n_0}{\sqrt{r^2-n_0^2}}\leq\frac{2n_0}{r},
\]
so
\[
\int_{2n_0}^\infty\frac{\pi-\eta_{\Omega_\ast}(r)}{r}\,\mathrm{d}r\leq2n_0\int_{2n_0}^{\infty}\frac{\mathrm{d}r}{r^2}+\int_{2n_0}^\infty\frac{\pi-\eta_{\Sigma'}(r)}{r}\,\mathrm{d}r=1+\int_{2n_0}^\infty\frac{\pi-\eta_{\Sigma'}(r)}{r}\,\mathrm{d}r.
\]

The right hand side converges if and only if it is convergent the series \[\sum_{k=1}^{\infty} \frac{a'_k-a'_{k-1}}{b'_k}=\frac{n_0+1}{b(n_0)}+\sum_{k=2}^{\infty}\frac{a_{k+1}-a_{k}}{b_k}=\frac{n_0+1}{b(n_0)}+\sum_{k=2}^{\infty}\frac{1}{b(n_0+k-1)}.\] So if $\sum_{j=n_0}^{\infty}1/b(j)<+\infty$, then the semigroup associated to $\Omega_\ast$ is of finite shift. This concludes the proof of $\ref{t_sh}\Leftrightarrow\ref{t_ser}$ of Theorem \ref{main}.

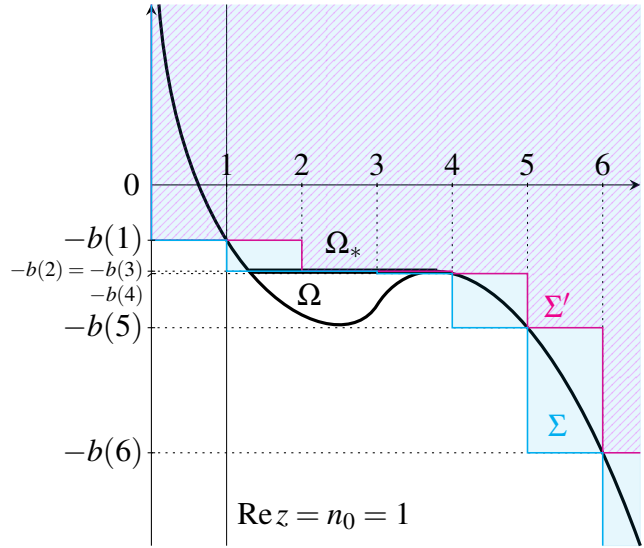
\begin{figure}[htb]
	\begin{tikzpicture}[yscale=0.8]
		\usetikzlibrary{patterns}
		\draw [-stealth] (0,-6)--(0,3);
		\draw [-stealth] (0,0)--(6.5,0);
		\draw [very thin] (1,3)--(1,-5.5) node [anchor=west] {$\mathrm{Re}\,z=n_0=1$}--(1,-6);
		\draw [dotted](1,0)--(1,-0.92)--(0,-0.92);
		\draw [dotted](2,0)--(2,-1.435);
		\draw [dotted](3,0)--(3,-1.435)--(0,-1.435);
		\draw [dotted](4,0)--(4,-1.475)--(0,-1.475);
		\draw [dotted](5,0)--(5,-2.375)--(0,-2.375);
		\draw [dotted](6,0)--(6,-4.455)--(0,-4.455);
		\draw [very thick](0.1,3) .. controls (0.4,-1.8) and (2.5,-3) .. (3,-2) .. controls (3.5,-1) and (5,-0.8) .. (6.5,-6);
		\draw [line width=2](1.3,-1.435)--(3.8,-1.435);
		\path [pattern color=magenta, pattern=north east lines, opacity=0.5] (0,3)--(0,-0.92)--(2,-0.92)--(2,-1.435)--(4,-1.435)--(4,-1.475)--(5,-1.475)--(5,-2.375)--
		(6,-2.375)--(6,-4.455)--(6.5,-4.455)--(6.5,3);
		\draw [magenta, semithick] (0,3)--(0,-0.92)--(2,-0.92)--(2,-1.435)--(4,-1.435)--(4,-1.475)--(5,-1.475)--(5,-2.375)--
		(6,-2.375)--(6,-4.455)--(6.5,-4.455);
		\draw [cyan, fill, opacity=0.1] (0,3)--(0,-0.92)--(1,-0.92)--(1,-1.435)--(3,-1.435)--(3,-1.475)--(4,-1.475)--(4,-2.375)--
		(5,-2.375)--(5,-4.455)--(6,-4.455)--(6,-6)--(6.5,-6)--(6.5,3);
		\draw [cyan, semithick] (0,3)--(0,-0.92)--(1,-0.92)--(1,-1.435)--(3,-1.435)--(3,-1.475)--(4,-1.475)--(4,-2.375)--
		(5,-2.375)--(5,-4.455)--(6,-4.455)--(6,-6);
		\node at (2.1,-1.81) {$\Omega$};
		\node at (2.56,-1) {$\Omega_\ast$};
		\draw (1,1pt) -- (1,-1pt) node[anchor=south] {$1$};
		\draw (2,1pt) -- (2,-1pt) node[anchor=south] {$2$};
		\draw (3,1pt) -- (3,-1pt) node[anchor=south] {$3$};
		\draw (4,1pt) -- (4,-1pt) node[anchor=south] {$4$};
		\draw (5,1pt) -- (5,-1pt) node[anchor=south] {$5$};
		\draw (6,1pt) -- (6,-1pt) node[anchor=south] {$6$};
		\node at (0,0)[anchor=east] {$0$};
		\draw (-0.05,-0.92) -- (0.05,-0.92) node[anchor=east] {$-b(1)$};
		\draw (-0.05,-1.435) -- (0.05,-1.435) node[anchor=east] {{\tiny $-b(2)=-b(3)$}};
		\draw (-0.05,-1.525) -- (0,-1.475) -- (0.05,-1.475) node[anchor=north east] {{\tiny $-b(4)$}};
		\draw (-0.05,-2.375) -- (0.05,-2.375) node[anchor=east] {$-b(5)$};
		\draw (-0.05,-4.455) -- (0.05,-4.455) node[anchor=east] {$-b(6)$};
		\node at (5.4,-4) {\textcolor{cyan}{$\Sigma$}};
		\node at (5.4,-2) {\textcolor{magenta}{$\Sigma'$}};
	\end{tikzpicture}
	\caption{The step domains $\Sigma$ and $\Sigma'$ associated to $\Omega_\ast$. Here $n_0=1$, $a_k=k$, $a_k'=k+1$ and $b_k=b_k'=b(k)$.}
	\label{step_omega}
\end{figure}

\begin{remark}
	The argument above can be easily generalized by taking any partition $\{a_k\}_{k\geq 1}$ of $[r_0,+\infty)$ such that $0<\inf_{k\geq 1}\,(a_{k+1}-a_{k})\leq\sup_{k\geq 1}\,(a_{k+1}-a_{k})<+\infty$: if $b(a_k):=\inf \{y\in\mathbb{R}_+:a_k-iy\notin \Omega_\ast\}$, then the shift of a semigroup associated to $\Omega$ is finite if and only if \[
	\sum_{k=1}^{\infty}\frac{1}{b(a_k)}<+\infty.
	\]
	If we only assume that there exists $C>0$ such that for all $k$ it is $a_{k+1}-a_k<C(a_k-a_{k-1})$, then by using the same argument one gets the same result where the series becomes \[\sum_{k=2}^{\infty}\frac{a_k-a_{k-1}}{b(a_k)}.\]
\end{remark}

%

\section*{Acknowledgments} I would like to thank professor Filippo Bracci for the help and all the improvements he suggested to me in writing down this article.
I would also like to thank the reviewer for the provided comments, corrections and further improvements.

\end{document}